\documentclass[a4paper,leqno]{amsart}

\usepackage{latexsym}
\usepackage[english]{babel}
\usepackage{fancyhdr}
\usepackage[mathscr]{eucal}
\usepackage{amsmath}
\usepackage{mathrsfs}
\usepackage{amsthm}
\usepackage{amsfonts}
\usepackage{amssymb}
\usepackage{amscd}
\usepackage{bbm}
\usepackage{graphicx}
\usepackage{graphics}
\usepackage{latexsym}
\usepackage{color}
\usepackage{calc}
\usepackage{enumitem}
\usepackage[linktocpage=true,colorlinks=true, linkcolor=blue, citecolor=blue, urlcolor=blue,hyperfootnotes=false]{hyperref}
\usepackage{cleveref}
\usepackage[left=3cm, right=3cm, top=3cm, bottom=3cm]{geometry}
\usepackage{cleveref}

\newcommand{\ud}{\mathrm{d}}

\newcommand{\nn}{\mathbf{n}}
\renewcommand{\tt}{\mathbf{t}}

\newcommand{\ii}{\mathrm{i}}

\newcommand{\D}{\mathcal{D}}

\newcommand{\Hmin}{H_{\operatorname{min}}}

\newcommand{\C}{\mathbb C}

\newcommand{\N}{\mathbb N}
\newcommand{\R}{\mathbb R}

\theoremstyle{plain}
\newtheorem{theorem}{Theorem}[section]

\newtheorem{corollary}[theorem]{Corollary}
\newtheorem{proposition}[theorem]{Proposition}

\theoremstyle{definition}

\newtheoremstyle{iremark}
	{6pt}
	{4pt}
	{\upshape}
	{0pt}
	{\itshape}
	{.}
	{5pt plus 1pt minus 1 pt}
	{\thmname{#1} \thmnumber{\itshape#2}\thmnote{(#3)}}

\theoremstyle{iremark}
\newtheorem{remark}[theorem]{Remark}
\newtheorem*{remark*}{Remark}

\numberwithin{equation}{section}

\begin{document}

\title{Dirac-Coulomb Operators with Infinite Mass Boundary
Conditions in Sectors}
\author[B.~Cassano]{Biagio Cassano}
\address[B.~Cassano]{Dipartimento di Matematica e Fisica, Università degli studi della Campania \\ Viale Lincoln, 5 \\ 
Caserta 81100  (Italy).}
\email{biagio.cassano@unicampania.it}
\author[M.~Gallone]{Matteo Gallone}
\address[M.~Gallone]{Mathematics Department ``F.~Enriques'', University of Milan \\ via C.~Saldini 50 \\ Milano 20133 (Italy).}
\email{matteo.gallone@unimi.it}
\author[F.~Pizzichillo]{Fabio Pizzichillo}
\address[F.~Pizzichillo]{Departamento de Matemática Aplicada y Ciencias de la Computación,
Universidad de Cantabria.\\
E.~T.~S.~de Ingenieros Industriales y de Telecomunicación,
Avenida de Los Castros 44,
Santander 39005 (Spain).}
\email{fabio.pizzichillo@unican.es}


\begin{abstract}
We investigate the properties of self-adjointness of a two-dimensional Dirac operator on an infinite sector with infinite mass boundary conditions and in presence of a Coulomb-type potential with the singularity placed on the vertex. In the general case, we prove the appropriate Dirac-Hardy inequality and exploit the Kato-Rellich theory. In the explicit case of a Coulomb potential, we describe the self-adjoint extensions for all the intensities of the potential relying on a radial decomposition in partial wave subspaces adapted to the infinite-mass boundary conditions. Finally, we integrate our results giving a description of the spectrum of these operators.

\end{abstract}

\date{\today}


\maketitle

\vspace{-1cm} 


\section{Introduction}


In this paper, we are interested in the two dimensional Dirac operator on an infinite sector, subject to {infinite mass} boundary conditions, in presence of a singular potential of Coulomb type, centred in the corner of the sector. 
The descriptions of the self-adjointness and spectral properties of the Dirac operator in a sector with infinite mass boundary conditions and of the Dirac operator with a Coulomb-type perturbation respectively are well understood, but a detailed analysis of the coupling of the two features is missing. It is interesting to describe their interaction since the two share the same singular nature: this is particularly evident in the case of an explicit Coulomb perturbation, see \Cref{rmk:singular}.

The Dirac operator was introduced in \cite{dirac} as the Hamiltonian generating the evolution of a relativistic particle with spin $\tfrac12$
on the whole three dimensional space and its analysis has been the subject of many investigations (see the monography \cite{thaller}). Parallel to that, it has found many other applications both for quark models in the atomic nucleus or, in its two-dimensional version, in the analysis of materials with Dirac fermion low-energy excitations, the most famous being certainly graphene, see e.g.~\cite{DiracMatRev} for a review. For these models, it is physically meaningful to consider the operator on some domain with boundaries to model either the confining property of quarks or the edge of a material. From a mathematical point of view, the introduction of boundaries requires that appropriate boundary conditions have to be considered in order to preserve self-adjointness.

The free Dirac operator in two spatial dimensions is given by the formal expression
\[
	D_0\;:=\;- \ii \boldsymbol{\sigma} \cdot \nabla+m
	\sigma_3=
	\begin{pmatrix}
		m									&	-\ii(\partial_{x_1}-\ii\partial_{x_2})\\
		-\ii(\partial_{x_1}-\ii\partial_{x_2})	&	-m	
	\end{pmatrix},
\]
where $m\geq 0$ is the mass of the particle and $\boldsymbol{\sigma}:=(\sigma_1,\sigma_2)$, being $\sigma_1,\sigma_2,\sigma_3$ the Pauli matrices
\[
	\sigma_1 \;=\; \begin{pmatrix}
		0 & 1 \\
		1 & 0
	\end{pmatrix} \, , \quad
	\sigma_2 \;=\; \begin{pmatrix}
		0 & -\ii \\
		\ii & 0
	\end{pmatrix} \, ,\quad 
	\sigma_3 \;=\; \begin{pmatrix}
		1 & 0 \\
		0 & -1
	\end{pmatrix} \, .
\]
The free Dirac operator in $\R^2$ is realised as a self-adjoint operator with domain $H^1(\R^2;\C^2)$. Its spectrum is purely essential and $\sigma(D_0) = \sigma_\mathrm{ess}(D_0) = (-\infty,-m] \cup [m,+\infty)$. In fact, $-\ii \boldsymbol{\sigma} \cdot \nabla$ is equivalent to the multiplication operator $\boldsymbol{ \sigma} \cdot \boldsymbol k$ through a Fourier transform, see \cite[Chapter 1]{thaller} for details.


In the analysis of boundary value problems on connected domains, one of the most interesting examples of boundary conditions for the applications is the one known as \emph{infinite mass boundary condition}. As its name suggests, it is given by considering the limit case of infinite mass outside the domain, see \cite{SV18,ALTMR19}. In detail, let $\Omega \subset \R^2$ be a connected domain such that its boundary $\partial\Omega$ is regular enough: we denote by $\nn$ the outward normal and by $\tt$ the tangent vector to $\partial \Omega$ chosen in such a way that $(\nn, \tt)$ is positively oriented.

The {infinite mass} boundary condition is defined  as
\begin{equation}\label{eq:infinite-mass-(BC)}
	\mathcal{B}_\nn \psi = \psi \ \text{ on } \partial \Omega
\end{equation}
where the  matrix $\mathcal{B}_\nn$ is 
\[
	\mathcal{B}_\nn \;=\; - \ii \sigma_3 \, \boldsymbol{\sigma} \cdot \nn \, .
\]
The regularity of $\Omega$ plays a fundamental role in this sort of problem. When $\Omega$ is $C^2$-regular the Dirac operator $D_0$ acting on the set of functions in $H^1(\Omega;\C^2)$ and that verifies \eqref{eq:infinite-mass-(BC)} is self-adjoint, see  \cite{benguria2017self}.
Such result is not anymore valid if we relax the regularity hypothesis of $\Omega$ and consider, for instance, domains with {corners}.
Let $\omega \in (0,2 \pi]$ and let $S_\omega$ be the two dimensional open sector of aperture $\omega$:
\begin{equation}\label{eq:defn.sector}
	S_\omega \;:=\; \big\{ (r\cos\theta,r\sin\theta) \in \mathbb{R}^2 \, : \, r > 0,\,0 < \theta < \omega \big\}.
\end{equation}
The problem of self-adjointness for the Dirac operator on $S_\omega$ with infinite mass boundary conditions is well understood: 
we resume some of the results from \cite{LTOB2018,PVDB2021} in the following theorem.
\begin{theorem}[{\cite{LTOB2018,PVDB2021}}]
Let $\omega \in (0,2\pi]$ and $S_\omega$ defined as in \eqref{eq:defn.sector}.
 Let $H_\omega$ be the operator 
\begin{equation}\label{eq:def-free-inf-mass}
\begin{split}
& H_\omega \psi := D_0 \psi, \\
& \D(H_\omega):=\left\lbrace 
			\psi\in H^1(S_\omega;\C^2):
			\mathcal{B}_\nn \psi = \psi \ \text{ on } \partial S_\omega
		\right\rbrace.
\end{split}
\end{equation}
Then:
\begin{enumerate}[label={(\roman*})]
	\item if $0<\omega\leq \pi$, $H_\omega$ is self-ajdoint;
	\item if $\pi<\omega \leq 2\pi$, $H_\omega$ admits infinite self-adjoint extensions and among them there exists a unique \emph{distinguished} one whose domain is included in the Sobolev space $H^{1/2}(S_\omega;\C^2)$.
\end{enumerate}
\end{theorem}
\begin{remark}
It is well known that when $\Omega$ is a bounded connected Lipschitz domain,
the boundary trace operator $\operatorname{tr}: H^1(\Omega)\to H^{1/2}(\partial \Omega)$ is well defined and bounded.
However, for bounded domains, $H^1$ is not the maximal domain for the differential expression $\boldsymbol{\sigma}\cdot\nabla$. For this reason, it is convenient to introduce
\[
	\mathcal{K} (\Omega) := \{ u \in L^2(\Omega;\C^2):\boldsymbol{\sigma} \cdot \nabla  u \in L^2(\Omega;\C^2) \}. 
\]
For this space a weaker notion of boundary trace can be given.  Indeed, when $\Omega$ is a \emph{curvilinear polygon}, then the operator $\boldsymbol{\sigma} \cdot \nn \operatorname{tr} : H^1(\Omega;\C^2) \to L^2(\partial \Omega,\C^2)$ extends to a bounded operator $T : \mathcal{K} (\Omega) \to H^{-1/2}(\partial \Omega;\C^2)$, see \cite[Lemma 2.3]{PVDB2021}.
Then, for $0<\omega \leq \pi$ the boundary condition $\mathcal{B}_\nn \psi = \psi$ on $\partial S_\omega$ in \eqref{eq:def-free-inf-mass} is intended in the sense of $H^{1/2}(\partial S_\omega;\C^2)$ while  for $\pi<\omega \leq 2\pi$ it has to be intended in the weaker sense of $H^{-1/2}(\partial \Omega;\C^2)$.
\end{remark}
We refer to \cite{PVDB2021} for the description of the more general quantum dot boundary conditions. Moreover, we report to \cite{CL20} for the analysis of the self-adjointness in the case of discontinuous infinite mass boundary conditions. Finally, the analogous problem in the three dimensional setting, namely the self-adjointness of the Dirac operator on a three dimensional cone with MIT bag boundary conditions, can be found in \cite{CL2022}. 

A strictly related topic is the description of Dirac operators with $\delta$-shell interactions, see e.g.~\cite{AMV14, AMV15, AMV16, BEHL16, BHSS21, 
HOP18, OV18}, the review papers~\cite{BEHL19, OP21} and the references therein, both in the two and three dimensional setting.
In fact, it is possible to describe Dirac operators on domains as Dirac operators coupled with $\delta$--shell interactions generating confinement, see \cite[Section 2.3]{CLMT21}. In this field of research there has been a big effort to lower the regularity assumptions for the boundary of the considered domain: we report to \cite{BHSS21,B22} and references therein for the general case of domains with Lipshitz boundary. In particular, in \cite{BHSS21} it is shown that the Dirac operator on a compact region with a corner admits a unique self-adjoint realisation whose domain is included in $H^{1/2}(S_\omega;\C^2)$, but in the particular case of the sector \cite{LTOB2018,PVDB2021} provide a more precise description of the domain. 
Also, we refer to \cite{FL21} for the description of the Dirac operator with Lorentz-scalar $\delta$-shell interactions supported on star-graphs, these including the analysis of Dirac operators with infinite mass boundary conditions on a sector.

As mentioned before, the analysis of two dimensional Dirac equation has attained a certain amount of interest from low-energy condensed matter physics. The succesful experimental isolation of a single plane of graphene provide an interesting test for non-perturbative quantum electrodynamics \cite{NatureQED}. In fact, depending on the material graphene is deposited on, electronic excitations can be well described in terms of a massive or a massless Dirac equation \cite{PereiraKotov}. These substrates interact with graphene resulting in effective potentials that may break symmetries of the lattice or generate gaps in the electronic spectrum. It is of particular importance the analysis of charged impurities as they play an important role in the transport properties of graphene. In this context, parameters entering Dirac equation translate in a small mass and strong interaction characterised by a large value of the effective fine structure constant, and one consequently expects that charge impurities may lead to phenomena beyond the perturbative description of quantum electrodynamics such as the ``vacuum polarisation'' \cite{VacuumPol1}. These problems are treated in the literature by adding a (critical or, even, supercritical) Coulomb potential to the Dirac equation. Therefore we can think of an excitation of graphene to be modelled by a Dirac equation in two spatial dimension with a $\nu/|x|$ potential centered in the position of a impurity \cite{VacuumPol2}. 

Naturally, the history of Dirac-Coulomb operator begins in the three dimensional setting as the very first motivation for its introduction was the analysis of the relativistic correction to the spectral lines of the hydrogen atom. We summarise its very interesting and rich history, see \cite{Gallone2017},\cite{CP2018} or \cite[Section 1.3]{ELS1} for more details. 
Rigorous analysis of Dirac-Coulomb Hamiltonian devoted to establishing its self-adjointness dates back to the early '50s in the works of Rellich \cite{RellichOrig} and Kato \cite{KatoOrig}; only in the early '70s it was recognised by several authors that the operator with purely Coulomb potential was essentially self-adjoint if and only if $|\nu| \leq \frac{\sqrt{3}}{2}$. In the same years, three (in principle) \emph{distinguished} self-adjoint extensions were built by Schmincke \cite{Schminke1}, Nenciu \cite{Nenciu} and W\"ust \cite{Wust} in the regime of higher nuclear charge ($\frac{\sqrt{3}}{2} < |\nu| < 1$) and, just before the end of the decade, it was recognised that the three extensions were in fact the \emph{same} \cite{Wust2}. It took several years to develop powerful Hardy-Dirac inequalities to push the definition of \emph{the} distinguished extension up to the value $|\nu|=1$ in \cite{EstebanLoss,ELS1}. 

In the regime $\frac{\sqrt{3}}{2} < |\nu|$ the Dirac-Coulomb operator in three spatial dimension is \emph{not} essentially self-adjoint, so the research focused on the classification of \emph{all} the self-adjoint realisations of the formal operator in this regime.  This result was achieved correctly in \cite{GM-Dirac-Selfadjointness} for $\frac{\sqrt{3}}{2} < |\nu|<1$, in \cite{Hogreve} for $|\nu|>1$ and in \cite{CP2018,CP2019}  for any $\nu \in \R$ with different techniques:  the adaptation of Kre\u{\i}n-Vi\v{s}ik-Birman-Grubb extension scheme, von Neumann extension theory and the restriction of the domain of the adjoint and boundary triplets respectively. More recently, Derezi\'{n}ski and Ruba \cite{Derezinski} classified and carefully analyzed closed extensions with complex-valued potentials. Leaving the realm of electrostatic fields generated by one point charge, we mention \cite{ELS1}, where the authors prove the existence of a distinguished self-adjoint extension for a generic (in a certain sense ``subcritical'') charge distribution. 

Let us emphasize that the analyses of \cite{GM-Dirac-Selfadjointness,CP2018,CP2019,Hogreve} rely on the angular decomposition of the Dirac-Coulomb operator, therefore the results can be translated directly to the two-dimensional case modifying only the eigenvalues of the angular momentum appearing in the radial operator. In particular, two-dimensional Dirac-Coulomb operator ceases to be essentially self-adjoint when $\nu \neq 0$ and the distinguished extension exists for $|\nu| < \frac{1}{2}$, see \cite{Cuenin}.


\subsection{Main results}
In this paper, we are interested in perturbing the Dirac operator $H_\omega$ on a sector with a potential of Coulomb-type.
To study self-adjointness, we use two different approaches: the Kato-Rellich theory and the explicit radial decomposition of the operator.
In the following, we assume $m=0$ without loss of generality, since a bounded perturbation does not influence such property.

The first and crucial tool in the analysis of perturbations of self-adjoint operators is the Kato-Rellich theorem.
Its use has a deep impact in physical applications being, for example, the key ingredient to prove self-adjointness of atomic Hamiltonians in non-relativistic quantum mechanics: in this setting, the inter-particle interaction is ``small'' in a certain sense with respect to the graph norm of an (essential) self-adjoint Hamiltonian. From an analytical point of view, the smallness of the inter-particle potential requires the validity of a Hardy inequality. 
Mimicking this approach for the Dirac operator on a sector, we present, as the first result of this paper, a {Hardy inequality} for the Dirac operator on sectors with infinite mass boundary conditions.
 \begin{theorem}[Dirac-Hardy inequality]\label{thm:hardy-on-sector}
 Let $\omega \in (0,2\pi]$, $S_\omega$ as in \eqref{eq:defn.sector} and let $H_\omega$ be defined as in \eqref{eq:def-free-inf-mass}.
For any $\psi\in \mathcal{D}(H_\omega)$ we have that:
		\begin{equation}\label{eq:hardy-on-sector}
			\int_{S_\omega} |\boldsymbol{\sigma}\cdot\nabla\psi|^2\,dx\geq \frac{(\pi-\omega)^2}{4\omega^2}\int_{S_\omega}\frac{|\psi|^2}{|x|^2}\,dx.
		\end{equation}
\end{theorem}
\begin{remark}
We underline that $\Vert -i\boldsymbol\sigma\cdot\nabla \psi \Vert_{L^2(\R^2;\C^2)} = \Vert \nabla \psi \Vert_{L^2(\R^2;\C^2)}$. Consequently, a non-trivial Hardy inequality as in \eqref{eq:hardy-on-sector} does not hold if we replace $S_\omega$ with the whole $\R^2$, since it does not hold for the gradient. \Cref{thm:hardy-on-sector} shows that the Hardy-type estimate \eqref{eq:hardy-on-sector} holds when we restrict to the domain $S_\omega$: such phenomena is known to happen also for the Hardy inequality for the gradient.
\end{remark}
Thanks to the Kato-Rellich and W\"{u}st Theorems \cite[Theorems X.12 and X.14]{reedsimon2}, \Cref{thm:hardy-on-sector} immediately implies the following stability result for the self-adjointness of $H_\omega$ under unbounded perturbations of Coulomb-type.
\begin{corollary}\label{cor:kato-rellich}
Let $\omega \in (0,\pi)$, $S_\omega$ as in \eqref{eq:defn.sector} and $H_\omega$ defined as in \eqref{eq:def-free-inf-mass}. Let ${V}:S_\omega \to \C^{2\times 2}$ such that ${V}(x)$ is Hermitian for a.a.~$x \in S_\omega$ and such that for some $\nu>0$
\[
\left\vert {V}(x)\right\vert \leq \frac{\nu}{|x|}, \quad \text{ for a.a. }x\in S_\omega,
\]
being $| V(x) |$ is the operator norm of the matrix $V(x) \in \C^{2\times 2}$.
Then:
\begin{enumerate}[label={(\roman*})]
\item if $\nu < \frac{\pi-\omega}{2\omega}$, $H_\omega + {V}$ is self-adjoint with 
$\mathcal{D}(H_\omega + {V}) = \mathcal{D} (H_\omega) $;
\item if $\nu = \frac{\pi-\omega}{2\omega}$, $H_\omega + {V}$ is essentially self-adjoint on $\mathcal{D}(H_\omega + {V}) = \mathcal{D} (H_\omega) $.
\end{enumerate}
\end{corollary}
\begin{remark}
	Hypotheses of \Cref{cor:kato-rellich} are satisfied for potentials that \emph{locally} diverge logarithmically. This is important because the divergence of two-dimensional electrostatic field in dimension 2 close to the charge is logarithmic. However, as discussed above, the interest for potentials of the type $1/|x|^\alpha$ arise when restricting a three dimensional model to a two dimensional effective one.
\end{remark}

In the particular case of the \emph{Coulomb potential}
\[
V(x):=\dfrac{\nu}{|x|} \mathbbm{1}_2, \quad \text{ for all } x \in S_\omega\setminus\{0\}
\]
we can provide a much more detailed description of the self-adjoint realisations of $D_0 + V$ exploiting the radial symmetry: in the following we extend (and improve) the results in \Cref{cor:kato-rellich} to any angle $\omega \in (0,2\pi]$ and $\nu \in \R$.
We define the \emph{minimal} operator $H_{\text{min}}$ as follows
\begin{equation}\label{eq:def-H_min}
	\begin{split}
		\mathcal{D}(H_{\text{min}}) &\;:=\; \left\{ u \in C^\infty_c\left(\,\overline{S_\omega}\setminus\{0\}\,;\,\mathbb{C}^2\,\right) \, : \, \mathcal{B}_\nn u = u  \text{ on } \partial S_\omega \right\},\\
	\Hmin u &\;:=\;(D_0+V) u\,,
	\end{split}
\end{equation}
The operator $\Hmin$ is {symmetric}, as can be seen with an explicit integration by parts.
%
Our next result is the classification of the self-adjoint extensions of the minimal operator $\Hmin$.
\begin{theorem}
\label{thm:self-ajdoint-realisations}
Let $\omega \in (0,2\pi]$, $S_\omega$ as in \eqref{eq:defn.sector} and  $\Hmin$ as in \eqref{eq:def-H_min}.
	Then:
	\begin{enumerate}[label=(\roman*)]  
			\item\label{item:self-adjointness} if 
		$\nu^2 \leq \frac{\pi^2 - \omega^2}{4\omega^2}$, the operator $\Hmin$ is essentially self-adjoint and 
		\[\mathcal{D}(\overline{\Hmin})  = \mathcal{D}(H_\omega) = \left\lbrace 
			\psi\in H^1(S_\omega;\C^2):
			\mathcal{B}_\nn \psi = \psi \ \text{ on } \partial S_\omega
		\right\rbrace;
		\]
		\item\label{item:many.extensions} if 
		$\nu^2 > \frac{\pi^2 - \omega^2}{4\omega^2}$, the operator $\Hmin$ has infinitely many self-adjoint extensions and there exists a one-to-one correspondence between the self-ajdoint extensions of $\Hmin$ and the space  $\mathcal{U}(d+1)$ of the unitary matrices on $\C^{d+1}$, being 
			\begin{equation}\label{eq:def-d}
				d:=\max\left\{k \in \N: k<\frac{\omega}{\pi} \sqrt{\nu^2+\frac{1}{4}}-\frac{1}{2} \right\} \, .
			\end{equation}
%
	\end{enumerate}
\end{theorem}
\begin{remark}
For $\omega \in (0,\pi)$, $\frac{(\pi - \omega)^2}{4\omega^2} < \frac{\pi^2 - \omega^2}{4\omega^2}$, so \Cref{thm:self-ajdoint-realisations} {\ref{item:self-adjointness}} gives a better result than \Cref{cor:kato-rellich}.
This is not surprising: already in the whole space a similar phenomenon occours. In fact, the Dirac-Hardy inequality \eqref{eq:hardy-on-sector} does not allow to exploit the peculiar matricial form of the Coulomb potential, and provides a more general (and weaker) result. For a discussion on this feature in the three dimensional setting, we refer to the introduction of \cite{CPV20}. 
\end{remark}

When $\frac{\pi^2 - \omega^2}{4\omega^2}<\nu^2\leq \frac{\pi^2}{4\omega^2}$, it is possible to select a \emph{distinguished} self-adjoint extension among all the self-adjoint extensions given in \Cref{thm:self-ajdoint-realisations} {\ref{item:many.extensions}}, requiring that the functions in its domain have the best possible behaviour in the origin. To this  purpose, for $w\in L^1_{\text{loc}}(\R^2)$ set 
	\[
		\mathcal{D}(w,B_1):=\{u\in L^2(\R^2): wu\in L^2(B_1)\},
	\]	
	where $B_1$ denotes the ball of radius $1$ centred at the origin.

\begin{theorem}\label{thm:distinghuished-self-ajdoint}
	Under the assumptions of \Cref{thm:self-ajdoint-realisations} assume moreover that $\frac{\pi^2 - \omega^2}{4\omega^2}<\nu^2$. Then:
\begin{enumerate}[label=(\roman*)]
\item if $\frac{\pi^2 - \omega^2}{4\omega^2}<\nu^2< \frac{\pi^2}{4\omega^2}$ there exists a unique self-adjoint extension $T^{(D)}$ of $\Hmin$ such that 

\[
\mathcal{D}(T^{(D)}) \subset \mathcal{D}(|x|^{-a},B_1),\quad \text{for all }  0\leq a <\frac{1}{2}+\sqrt{\frac{\pi^2}{4\omega^2}-\nu^2}.
\]
			Thus $T^{(D)}$ is the distinguished extension.

\item if $\frac{\pi^2}{4\omega^2} = \nu^2$ there exists a unique self-adjoint extension $T^{(D)}$ of $\Hmin$ such that 
\begin{equation}\label{eq:condition.distinguished.critical}
\mathcal{D}(T^{(D)}) \subset \mathcal{D}((|x|^{a}\log^2|x|)^{-1} ,B_1),\quad \text{for all }  0\leq a \leq \frac{1}{2}.
\end{equation}
			Thus $T^{(D)}$ is the distinguished extension.

\item
if $\frac{\pi^2}{4\omega^2} < \nu^2$, there exist infinite extensions $T$ of $\Hmin$ such that 
\begin{equation}
\mathcal{D}(T^{(D)}) \subset \mathcal{D}((|x|^{a}\log^2|x|)^{-1} ,B_1),\quad \text{for all }  0\leq a \leq \frac{1}{2}.
\end{equation}
Thus $\Hmin$ does not have any distinguished extension.
\end{enumerate}
\end{theorem}
\begin{remark}\label{rem:sobolev-regularity}
If $\frac{\pi^2 - \omega^2}{4\omega^2}<\nu^2< \frac{\pi^2}{4\omega^2}$, the distinguished extension $T^{(D)}$ can be characterised in terms of Sobolev regularity. Indeed, combining \Cref{thm:distinghuished-self-ajdoint} with \cite[Theorem 1.4.5.3]{grisvard}, $T^{(D)}$ is the unique extension of $\Hmin$ that verifies  
\[
\mathcal{D}(T^{(D)}) \subset H^{s}(S_\omega;\C^2),\quad\text{for}\ s<\frac{1}{2}+\sqrt{\frac{\pi^2}{4\omega^2}-\nu^2}.
\]
Nevertheless, this characterisation fails in the case $\nu^2\geq \frac{\pi^2}{4\omega^2}$, where one can see that there exists infinite self-adjoint extension verifying the property:
\[
\mathcal{D}(T) \subset H^{s}(S_\omega;\C^2),\quad\text{for}\ s<\frac{1}{2}.
\]
\end{remark}

Having extablished the self-adjointness of Coulomb-type perturbations of $H_\omega$, we turn our analysis to a description of their spectrum. In the following part of the introduction we consider in general $m \geq 0$ in the definition of $H_\omega$.
Our first result in this directions complements \Cref{cor:kato-rellich} and \Cref{thm:self-ajdoint-realisations} investigating the stability of the essential spectrum of $H_\omega$ under general Coulomb-type perturbations.
\begin{theorem}\label{thm:ess-spect}
	Let $\Hmin$ be the operator defined in \eqref{eq:def-H_min} and let $T$ be any self-adjoint extension of $\Hmin$. Then

\[
				\sigma_{\mathrm{ess}}(T)=(-\infty,-m]\cup[m,+\infty) \, .
\]
	Moreover, when $\frac{\pi^2 - \omega^2}{4\omega^2}<\nu^2$, that is when $\Hmin$ is not essentially self-adjoint, for any $\lambda \in (-m,m)$ there exists $T$ a self-adjoint extension of $\Hmin$ for which $\lambda $ is an eigenvalue.
\end{theorem}

\begin{remark}
The result of Theorem \ref{thm:ess-spect} can be translated immediately to the case of singular potentials as $V$ verifying the hypotesis of \Cref{cor:kato-rellich}. In this case the self-adjoint realisation $H_\omega+V$ has $\sigma_{\mathrm{ess}}(H_\omega+V)=(-\infty,-m] \cup [m,+\infty)$.
\end{remark}


The infinite mass boundary conditions prevent the massive operator $H_\omega$ to be diagonalised by the unitary transformation of Proposition \ref{prop:UnitaryTransformation}. This makes difficult to mimic the computation of eigenvalues or the characterisation of the discrete spectrum of Dirac operators with explicit Coulomb potentials as in \cite{GM-Dirac-Eigenvalues}.
So we can not provide further details in the case of an explicit Coulomb potential using the radial decomposition.

\section{The radial operator and proofs of \texorpdfstring{\Cref{thm:self-ajdoint-realisations}}{Theorem 1.7} and \texorpdfstring{\Cref{thm:distinghuished-self-ajdoint}}{Theorem 1.9}}
\label{sec:radial-operator}
To prove \Cref{thm:self-ajdoint-realisations} and \Cref{thm:distinghuished-self-ajdoint}, we decompose the Dirac operators $H_\text{min}$ for $m=0$ in the direct sum of one dimensional Dirac operators on the half-line. We introduce some notations: for $k\in\N$, set
\begin{equation}\label{eq:def-lambda_k}
	\lambda_k\;:=\;\frac{(2k+1) \pi}{2 \omega}\,,
\end{equation}
and let $d_{\nu,k}$ be the differential expression
\begin{equation}\label{eq:def-d_k}
	d_{\nu,k}\;:=
	\begin{pmatrix}
		\frac{\nu}{r} & -\partial_r-\frac{\lambda_k}{r} \\
		\partial_r -\frac{\lambda_k}{r} & \frac{\nu}{r}
	\end{pmatrix} \, .
\end{equation}
We define the following Dirac operators on the half-line:
	\begin{equation}\label{eq:def-h_k}
	\begin{split}
	& \D(h_{\nu,k}):= C^\infty_c\left ((0,+\infty);\C^2\right)\\
	& h_{\nu,k} u:=d_{\nu,k} u.
	\end{split}
	\end{equation}
%

\begin{proposition}\label{prop:radial.decomposition}
Let $\nu \in \R$,  $\omega \in (0,2\pi]$, $S_\omega$ defined as in \eqref{eq:defn.sector}, $H_{\mathrm{min}}$ defined as in \eqref{eq:def-H_min} and for all $k \in \N$ let $h_{\nu,k}$  be defined as in \eqref{eq:def-h_k}. Then
\[
		 H_{\operatorname{min}} \;\cong\; \bigoplus_{k \in \N} h_{\nu,k}\
\]
	where `` $\cong$ '' means that the operators are unitarily equivalent. 
\end{proposition}
\begin{remark}\label{rmk:singular}
The \emph{partial wave subspace} decomposition of $H_\omega$ (given by \Cref{prop:radial.decomposition} for $\nu=0$) leads us to the analysis of an orthogonal sum of half-line Dirac operators perturbed by a off-diagonal Coulomb potentials, expressed via the eigenvalues of the \emph{spin-orbit operator}. This is the reason why we perturb the operator $H_\omega$ with an external Coulomb-potential. 
Roughly speaking, the presence of a corner in the origin and
the presence of an external Coulomb perturbation have the same singular nature: they both imply the presence of a singular term of order $\sim 1/r$ in the radial operators $h_{\nu,k}$.
\end{remark}

The proof of \Cref{prop:radial.decomposition} exploits the radial symmetry of the problem and takes advantage of the decomposition of the Hilbert space $L^2(S_\omega;\C^2)$ in the \emph{partial wave subspaces}. We omit these details here and we leave the proof of \Cref{prop:radial.decomposition} to \Cref{sec:appendix-angular}.

Thanks to \Cref{prop:radial.decomposition}, the proof of  \Cref{thm:self-ajdoint-realisations} follows from the analysis of the same properties on the reduced operators $h_{\nu,k}$.
For any $k \in \N$, the operator $h_{\nu,k}$ is a \emph{radial} Dirac operator and it has been studied in several works. Indeed,  the operator $h_{\nu,k}$ is precisely in the form of the one defined in \cite[Equation (2.19)]{CP2018} with $m=\lambda=\mu=0$ and $k_j=-\lambda_k$.

In the following proposition we study its self-adjointness.

\begin{proposition}\label{prop:description.extensions}
Let $k \in \N$ and let $h_{\nu,k}$ be defined as in \eqref{eq:def-h_k}. Let  moreover \begin{equation}\label{eq:def-chi}
	\chi\in C^\infty_c(\R;[0,1])\ \text{such that}\ \chi'(r)\leq 0\ \text{and}\
	\chi(r)
	=
	\begin{cases}
		1&\text{for}\ r\leq 1\\
		0&\text{for}\ r\geq 2
	\end{cases}.
\end{equation}
The following hold:
	\begin{enumerate}[label=(\roman*)]\label{thm:self-ajdoint-realisations-rad}
  		\item\label{item:def-self-adjoint-realisation-radial-ottimo>1/4}
  		If $\lambda_k^2-\nu^2\geq 1/4$, then $h_{\nu,k}$ is essentially self-ajdoint. 
		Moreover
\[
\begin{split}
		& \mathcal{D}(\overline{h_{\nu,k}})= H_0^1((0,+\infty);\C^2)  \quad \text{ if }\lambda_k^2-\nu^2 >\frac14, \\
		& \mathcal{D}(\overline{h_{\nu,k}}) \supsetneq H_0^1((0,+\infty);\C^2)  \quad \text{ if }\lambda_k^2-\nu^2 = \frac14 ;
\end{split}
\]
    	\item\label{item:def-self-adjoint-realisation-radial-sottocritico}
    	If $0<\lambda_k^2-\nu^2<1/4$ then $h_{\nu,k}$ is not essentially self-adjoint and it admits a one parameter family of self-adjoint extensions $\{t^{(\alpha)}_{\nu,k}\}_{\alpha\in [0,\pi)}$ such that
\[
    		\D({t^{(\alpha)}_{\nu,k}})=\operatorname{span}\big\{u_{\nu,k}^{(\alpha)}\big\}+ H_0^1((0,+\infty);\C^2),
\]
    	being
\[
    		u_{\nu,k}^{(\alpha)}(r)=P_{\nu,k}\cdot
    		\begin{pmatrix}
    		\cos(\alpha) r^{\sqrt{\lambda_k^2-\nu^2}})\\
    		\sin(\alpha) r^{-\sqrt{\lambda_k^2-\nu^2}})
    	\end{pmatrix}\chi(r)
	\quad \text{ for }r>0,
\]
    	and where $P_{\nu,k} \in \R^{2\times2}$ is the invertible matrix
\[
		P_{\nu,k}:= \frac{1}{2\sqrt{\lambda_k^2-\nu^2}\, (-\lambda_k-\sqrt{\lambda_k^2-\nu^2})}
		\begin{pmatrix}
 			-\lambda_k-\sqrt{\lambda_k^2-\nu^2} 	& \nu \\
			-\nu        						& \lambda_k+ \sqrt{\lambda_k^2-\nu^2}
 			\end{pmatrix}.
\]
		
		\item\label{item:def-self-adjoint-realisation-radial-critico}
		If $\lambda_k^2-\nu^2=0$ then $h_{\nu,k}$ is not essentially self-ajdoint and it admits a one parameter family of self-adjoint extension $\{t^{(\alpha)}_{\nu,k}\}_{\alpha\in [0,\pi)}$ such that
\[
    		\D({t^{(\alpha)}_{\nu,k}})=\operatorname{span}\big\{u_{\nu,k}^{(\alpha)}\big\}+H_0^1((0,+\infty);\C^2),
\]
    	being
\[
    		u_{\nu,k}^{(\alpha)}(r)=(Q_{\nu,k}\log(r)+\mathbbm{1}_2)\cdot
    		\begin{pmatrix}
    		\cos(\alpha) \\
    		\sin(\alpha)
    	\end{pmatrix}\chi(r)
	\quad \text{ for }r>0,
\]
    	where $Q_{\nu,k}\in \R^{2\times2}$, $Q_{\nu,k}^2=\mathbbm{0}_{2\times 2}$ is defined as follows
\[
			Q_{\nu,k}:=  			
  			\begin{pmatrix}
 				\lambda_k	& -\nu \\
				\nu       	& -\lambda_k
 			\end{pmatrix}.
\]
 		\item\label{item:def-self-adjoint-realisation-radial-sopracritico}
 		If $\lambda_k^2-\nu^2<0$ then $h_{\nu,k}$ is not essentially self-ajdoint and it admits a one parameter family of self adjoint extensions $\{t^{(\alpha)}_{\nu,k}\}_{\alpha\in [0,\pi)}$ such that
\[
    		\D({t^{(\alpha)}_{\nu,k}})=\operatorname{span}\big\{u_{\nu,k}^{(\alpha)}\big\}+H_0^1((0,+\infty);\C^2),
\]
    	being
\[
    		u_{\nu,k}^{(\alpha)}(r)=R_{\nu,k}\cdot
    		\begin{pmatrix}
    		\cos(\alpha) r^{\ii \sqrt{\nu^2-\lambda_k^2}})\\
    		\sin(\alpha) r^{-\ii\sqrt{\nu^2-\lambda_k^2}})
    	\end{pmatrix}\chi(r)
	\quad \text{ for }r>0
\]
    	where and $R_{\nu,k} \in \R^{2\times2}$ is the invertible matrix
\[
			R_{\nu,k}:= \frac{1}{2\ii\sqrt{\nu^2-\lambda_k^2}\, (-\lambda_k-\ii\sqrt{\nu^2-\lambda_k^2})}
		\begin{pmatrix}
 			-\lambda_k-\ii\sqrt{\nu^2-\lambda_k^2} 	& \nu \\
			-\nu        						& \lambda_k +\ii \sqrt{\nu^2-\lambda_k^2}
 			\end{pmatrix}.
\]
  \end{enumerate}
\end{proposition}
Before giving the proof of \Cref{thm:self-ajdoint-realisations-rad}, let  us now first characterise the closure of $h_{\nu,k}$.
\begin{proposition}\label{prop:closure}
Let $h_{\nu,k}$ be defined as in \eqref{eq:def-d_k}. Then
\[
		\mathcal{D}(\overline{h_{\nu,k}})=
		\begin{cases}
			H^1_0\left((0,+\infty);\mathbb{C}^2\right),&\text{if}\  \lambda_k^2-\nu^2 \neq \frac{1}{4};\\
			\left\{u\in L^2\left((0,+\infty);\mathbb{C}^2\right): u_1'-\frac{u_1}{r}\in L^2(0,+\infty), u_2\in H^1_0(0,+\infty)\right\}, &\text{if}\ \lambda_k^2-\nu^2 = \frac{1}{4}.
		\end{cases}
\]
\end{proposition}
\begin{proof}
	We denote by $C$ any positive constant.	
	Applying the one-dimensional Hardy inequality (see \cite[Proposition 2.4]{CP2018}):
	\[
	\int_0^\infty|f'(r)|^2\,dr\geq\int_0^\infty\frac{1}{4}\frac{|f(r)|^2}{r^2}\,dr\quad
	\text{for}\, f\in C^\infty_c(0,+\infty),
	\]
	we have that
	\[
	||h_{\nu,k}u||_{L^2}\leq C ||\partial_r u||_{L^2}, \quad\text{for}\ u\in\,\D(h_{\nu,k}).
	\]
	This implies that 
	\[
	H^1_0((0,+\infty);\mathbb{C}^2)\subset\mathcal{D}(\overline{h_{\nu,k}}).
	\]
	Set 
\[
	\tilde{\lambda}_k:=
	\begin{cases}
	\sqrt{\lambda _k^2-\nu ^2}&\text{if}\ \lambda_k^2-\nu^2\geq 0;\\
	\ii\sqrt{\nu^2-\lambda_k}&\text{otherwise},
	\end{cases}
\]
	and let $M_1,M_2$ be the matrices defined as 
\[
	M_1:=
	\begin{pmatrix}
	  \nu  & \lambda _k+\tilde{\lambda}_k \\
 \lambda _k+\tilde{\lambda}_k & \nu 
	\end{pmatrix}\quad
	M_2:=\begin{pmatrix}
	 -\nu  & \lambda _k+\tilde{\lambda}_k \\
 \lambda _k+\tilde{\lambda}_k & -\nu 
	\end{pmatrix}.
\]
	We get with an easy computation that
	\begin{equation}\label{eq:diag-radial}
	M_1\cdot
	\begin{pmatrix}
		\frac{\nu}{r} & -\partial_r-\frac{\lambda_k}{r} \\
		\partial_r -\frac{\lambda_k}{r} & \frac{\nu}{r}
	\end{pmatrix}=
	\begin{pmatrix}
	0&-\partial_r-\frac{\tilde{\lambda}_k}{r}\\
	\partial_r-\frac{\tilde{\lambda}_k}{r}&0
	\end{pmatrix}
	\cdot M_2.
	\end{equation}
	Let $\tilde{h}_{\nu,k}$ be the operator defined as
\[
	\D(\tilde{h}_{\nu,k})=C^\infty_c\left((0,+\infty);\C^2\right)\quad \tilde{h}_{\nu,k}u= \begin{pmatrix}
	0&-\partial_r-\frac{\tilde{\lambda}_k}{r}\\
	\partial_r-\frac{\tilde{\lambda}_k}{r}&0
	\end{pmatrix}\cdot 
	\begin{pmatrix}
	u_1\\
	u_2
	\end{pmatrix}.
\]
	
	For $\lambda_k^2-\nu^2\neq 0$, the matrices $M_1$ and $M_2$ are invertible. This means that the graph norm induced by $h_{\nu,k}$ is equivalent to the graph norm induced by $\tilde{h}_{\nu,k}$.
	We get the desired result applying the results of \cite[Appendix A]{FL21}. Although in this work the authors assume that $\tilde{\lambda}_k$ is real, the same approach can be used for purely imaginary constants. 

	Let us now assume $\lambda_k^2-\nu^2=0$. In this case, one can easily see that 
\[
	M_1\cdot M_2=M_2\cdot M_1=0\quad\text{and}\quad
	M_1+M_2=2\lambda_k\begin{pmatrix}
	0&1\\
	1&0
	\end{pmatrix}.
\]
	Then for any  $u\in\D(h_{\nu,k})$, thanks to \eqref{eq:diag-radial} and by the one-dimensional Hardy-inequality we have that:
	\[
	\begin{split}
	||h_{\nu,k} u||_{L^2}^2&=\frac{1}{4\lambda_k^2}||M_1 h_{\nu,k}||_{L^2}^2+\frac{1}{4\lambda_k^2}||M_2 h_{\nu,k}||_{L^2}^2
	\geq \frac{1}{4\lambda_k^2} ||\tilde{h}_{\nu,k}M_2 u||_{L^2}^2=\frac{1}{4\lambda_k^2}||\partial_r(M_2 u)||_{L^2}^2\\
	&
	\geq C \left\Vert \frac{M_2}{r} u \right\Vert_{L^2}^2.
	\end{split}
	\]		
	Thanks to this, and by definition of $h_{\nu,k}$ we have that
	\[
	||\partial_r u||_{L^2}=\left\Vert \begin{pmatrix}
0&1\\
-1&0
\end{pmatrix}\cdot
\left(
h_{\nu,k}+\frac{M_2}{r}\right)u\right\Vert_{L^2}
\leq C ||h_{\nu,k}u||_{L^2}.
	\]
%
This proves that 
\[
\mathcal{D}(\overline{h_{\nu,k}})\subset H^1_0((0,+\infty);\mathbb{C}^2),
\]
and it concludes the proof.
\end{proof}

\begin{proof}[Proof of \Cref{thm:self-ajdoint-realisations-rad}]
	Let us denote by $h_{\nu,k}^*$ the adjoint operator of $h_{\nu,k}$. Then, by definition we have that
\[
	\begin{split}
	& \D(h_{\nu,k}^*):= \left\{u\in L^2\left((0,+\infty);\C^2\right):d_{\nu,k}u\in L^2\left((0,+\infty);\C^2\right)\right\}\\
	& h_{\nu,k}^* u:=d_{\nu,k} u,
	\end{split}
\]
	where $d_{\nu,k} u$ has to be read in the distributional sense.

	From the analysis of \cite{CP2018} it turns out that $\delta:=\lambda_k^2-\nu^2$ is the parameter ruling essential self-adjointness of $h_{\nu,k}$. Indeed, from \cite[Theorem 1.1]{CP2018} we know that $h_{\nu,k}$ is essentially self-adjoint if $\delta \geq \tfrac{1}{4}$. This consideration, together with the explicit characterisation of the closure of \Cref{prop:closure} yields precisely \emph{\ref{item:def-self-adjoint-realisation-radial-ottimo>1/4}}.
	
	Let us now assume $0<\delta<\frac{1}{4}$. By  \cite[Theorem 1.2 \emph{(i)}]{CP2018} the operator $h_{\nu,k}$ is not essentially self-adjoint and it admits a one-parameter family of self-adjoint extensions $\left\{t^{(\alpha)}_{
	\nu,k}\right\}_{\alpha\in[0,\pi)}$. Moreover $u\in\mathcal{D}(t^{(\alpha)}_{	\nu,k})$ if and only if $u\in \mathcal{D}(h_{\nu,k}^*)$ and there exists $(A^+,A^-)\in\C^2$ such that
		\begin{gather}\label{eq:comb-lin-sin-cos}
		A^+\sin(\alpha)+A^-\cos(\alpha)=0,\\
		\label{eq:as-domain}
		u(r)=P_{\nu,k} 
		\begin{pmatrix}
			A^+ r^{\sqrt{\delta}} \\ 
			A^- r^{-\sqrt{\delta}}
		\end{pmatrix} 
		+o(r^{1/2})\quad\text{for}\ r\to 0.
	\end{gather}
	Let us decompose
\[
		u(r) = P_{\nu,k} \begin{pmatrix}
			A^+ r^{\sqrt\delta} \\
			A^- r^{-\sqrt\delta}
		\end{pmatrix} \chi(r) + \left[ u(r)-P_{\nu,k} \begin{pmatrix}
			A^+ r^{\sqrt\delta} \\
			A^- r^{-\sqrt\delta}
		\end{pmatrix} \chi(r) \right] \,=:v(r)+w(r).
\]
	Thanks to \eqref{eq:comb-lin-sin-cos}, we have that $v\in\operatorname{span}\left(u_{\nu,k}^{(\alpha)}\right)$. 
	
	Let us focus on $u_2$. By definition, $w\in \D(h_{\nu,k}^*)$ and. Moreover, by \eqref{eq:as-domain} $w(r)=o(r^{1/2})$ as $r \to 0$. Thanks to this and applying \cite[Equation (3.3)]{CP2018} we have that 
\[
		\int_0^\infty\frac{|w(r)|^2}{r^2}\,dr<+\infty.
\]
	For this reason
	\[
	\begin{pmatrix}
	0 & -\partial_r\\
	\partial_r & 0
	\end{pmatrix}w=
	h_{\nu,k}^*w-
	\begin{pmatrix}
		\nu & -\lambda_k \\
		-\lambda_k& \nu
	\end{pmatrix} \,\frac{w}{r} \in L^2\left((0,+\infty);\C^2\right),
	\]
	that proves that $w\in H^1_0\left((0,+\infty);\C^2\right)$. This implies that
\[
    		\D({t^{(\alpha)}_{\nu,k}})\subset\operatorname{span}\big\{u_{\nu,k}^{(\alpha)}\big\}+ H_0^1((0,+\infty);\C^2).
\]
    The other inclusion is obvious and this concludes the proof of \emph{\ref{item:def-self-adjoint-realisation-radial-sottocritico}}.

Last two points are proved analogously, but they rely on \cite[Theorem 1.2 \emph{(ii)} and Theorem 1.3]{CP2018} respectively. 	
\end{proof}

We are now ready to prove \Cref{thm:self-ajdoint-realisations} and \Cref{thm:distinghuished-self-ajdoint}.
\begin{proof}[Proof of \Cref{thm:self-ajdoint-realisations}]
Looking for self-adjoint extensions of $H_{\mathrm{min}}$ is equivalent to looking for self-adjoint extensions of the massless problem (as their difference is a bounded operator). The latter can be conveniently expressed in the unitary equivalent form of a direct sum (Proposition \ref{prop:radial.decomposition}) which can be exploited for the computation of deficiency indices (see \cite[Section 1.6]{GM-book}). From Proposition \ref{prop:description.extensions} we know that the deficiency indices of $h_{\nu,k}$ are
\[
	\dim \ker(h_{\nu,k}^*\pm i) \;=\; \begin{cases}
		0 \qquad \text{if} \qquad \lambda_k^2-\nu^2 \geq \frac{1}{4}, \\
		1 \qquad \text{if} \qquad \lambda_k^2-\nu^2 < \frac{1}{4}.
	\end{cases}
\]
It follows immediately that if $\lambda_0^2-\nu^2=\frac{\pi^2}{4 \omega^2}-\nu^2 \geq \frac{1}{4}$ then $\dim \ker (h_{\nu,k}^* \pm i) = 0$ for all $k \in \mathbb{N}$ and thus, by the basic criterion of essential self-adjointness \cite[Corollary to Theorem VIII.3]{reedsimon2}, $H_{\mathrm{min}}$ is essentially self-adjoint.

Using the characterisation of \Cref{prop:closure} one completes the proof of \emph{\ref{item:self-adjointness}}. For any $k \in \mathbb{N}$ such that $\lambda_k^2-\nu^2 < \frac{1}{4}$, the operator $h_{\nu,k}$ is not essentially self-adjoint. Since each non essentially self-adjoint radial operator contributes increasing the definciency indices of $1$, deficiency index of the total operator will be equal to the number of non-self-adjoint radial operators. For fixed $\nu$ and $\omega$, the condition on $k$ for $h_{\nu,k}$ being non essentially self-adjoint is
\[
	k < \frac{\omega}{\pi} \sqrt{\nu^2+\frac{1}{4}} - \frac{1}{2} \, .
\]
Calling $d$ the maximum of such $k$, deficiency indices of $H_{\mathrm{min}}$ are $d+1$ and then, from von Neumann's theorem of self-adjoint extensions \cite[Theorem X.2]{reedsimon2} we get the thesis.
%
\end{proof}

The proof of \Cref{thm:distinghuished-self-ajdoint} follows from \Cref{prop:radial.decomposition}, \Cref{prop:description.extensions} and the analysis of analogous properties on the reduced operators $h_{\nu,k}$. 
We study them in the following proposition.
\begin{proposition}
\label{thm:distinghuished-self-ajdoint-rad}
	Let $h_{\nu,k}$ and be defined as in \eqref{eq:def-h_k} and assume that $\lambda_k^2-\nu^2<1/4$. 
	Let $\{t^{(\alpha)}_{\nu,k})\}_{\alpha\in [0,\pi)}$ be the family of all the self-ajdoint extensions of $h_{\nu,k}$ defined respectively as in \Cref{thm:self-ajdoint-realisations-rad} \ref{item:def-self-adjoint-realisation-radial-sottocritico}--\ref{item:def-self-adjoint-realisation-radial-sopracritico}.
	Then:
	\begin{enumerate}[label=(\roman*)]
		\item\label{item:choose.distinguished} if $0 <\lambda_k^2-\nu^2$, $h_{\nu,k}$ admits a unique self-adjoint extension $t^{(D)}_{\nu,k}$ that verifies the property
		\begin{equation}\label{eq:condition.radial}
			\D(t^{(D)}_{\nu,k})\subset
			\mathcal{D}(r^{-a}\chi_{\{r\leq 1\}}),\quad \text{for all }  0\leq a <\frac{1}{2}+\sqrt{\lambda_k^2-\nu^2};
		\end{equation}	
		\item\label{item:choose.distinguished.critical} if $0 = \lambda_k^2-\nu^2 $, $h_{\nu,k}$ admits a unique self-adjoint extension $t^{(D)}_{\nu,k}$ that verifies the property
		\begin{equation}\label{eq:condition.critical.radial}
			\D(t^{(D)}_{\nu,k})\subset
			\mathcal{D}((r^{a}\log^2 r)^{-1} \chi_{\{r\leq 1\}}),\quad \text{for all }  0\leq a \leq \frac{1}{2};
		\end{equation}	
 		\item\label{item:supercritical} $\lambda_k^2-\nu^2<0$, all the self-adjoint extension $t^{(\alpha)}_{\nu,k}$ of $h_{\nu,k}$ verify
		\begin{equation}\label{eq:condition.supercritical.radial}
			\D(t^{(\alpha)}_{\nu,k})\subset
						\mathcal{D}((r^{a}\log^2 r)^{-1} \chi_{\{r\leq 1\}}),\quad \text{for all }  0\leq a \leq \frac{1}{2}.
		\end{equation}		
	\end{enumerate}
\end{proposition}
\begin{proof}
Thanks to \Cref{prop:description.extensions}, with explicit computation we have that condition \eqref{eq:condition.radial} 
is verified for $t_k^{(\alpha)}$  if and only if $\alpha = 0$. Analogously condition \eqref{eq:condition.critical.radial}  is verified if and only if $\alpha = -(\text{sign } \nu) \pi/4$. Finally, \eqref{eq:condition.supercritical.radial} is verified for any $\alpha\in[0,\pi)$.
\end{proof}

\begin{proof}[Proof of \Cref{thm:distinghuished-self-ajdoint}]
Let us firstly assume $\frac{\pi^2 - \omega^2}{4\omega^2}< \nu^2< \frac{\pi^2}{4\omega^2}$. 
This implies that $d$ defined in \eqref{eq:def-d} is equal to $0$. Thus, by \Cref{thm:self-ajdoint-realisations}, the operator $\Hmin$ has infinitely many self-adjoint extensions and there exists a one-to-one correspondence between the self-ajdoint extensions of $\Hmin$ and the space  $\mathcal{U}(1)\sim[0,\pi)$.
Since $0 \leq \lambda_0 - \nu^2 <1/4 $, and $\lambda_k^2 - \nu^2 \geq 1/4$ for all $k \in \N\setminus \{0\}$, thanks to \Cref{thm:self-ajdoint-realisations-rad} we have that for any $\alpha
\in[0,\pi)$:
\[
T^{(\alpha)} \cong t_0^{(\alpha)} \oplus \bigoplus_{k \in \N \setminus \{0\}} \overline{h_{\nu,k}}.
\]
From \Cref{thm:distinghuished-self-ajdoint-rad} \emph{\ref{item:choose.distinguished}}, we have that the self-adjoint realisation defined through the unitary map in \eqref{prop:UnitaryTransformation} as
\[
T^{(D)} :\cong t_0^{(D)} \oplus \bigoplus_{k \in \N \setminus \{0\}} \overline{h_{\nu,k}}
\]
is the unique one whose domain in included in $\mathcal{D}(|x|^{-a},B_1)$.
Last two points are proved analogously, but they rely on \Cref{thm:distinghuished-self-ajdoint-rad} \emph{\ref{item:choose.distinguished.critical}},  \emph{\ref{item:supercritical}} respectively. 	
\end{proof}

\section{Spectral Properties, Hardy-Dirac and proof of \texorpdfstring{\Cref{thm:ess-spect}}{Theorem 1.11}}
\begin{proof}[Proof of \Cref{thm:hardy-on-sector}]
Let $\psi\in C^\infty_c\left(\overline{S_\omega}\setminus\{0\};\C^2\right) \subset \mathcal{D}(H_\omega)$.
Thanks to \Cref{prop:UnitaryTransformation} for $\nu = 0$, there exist $u_k^+, u_k^- \in C^\infty_c\left ((0,+\infty);\C^2\right)$, $k \in \N$, such that 
\eqref{eq:rappr.psi.radial} holds true. 
Thanks to \eqref{eq:compute.HV} and \eqref{eq:compute.HV.2},  we explicitly compute
	\begin{equation}\label{eq:dec-sigma-grad-radial}
	\begin{split}
		\int_{S_\omega} |\boldsymbol\sigma\cdot \nabla\psi|^2\,dx
			&=
			\sum_{k\in \N} \Bigg[
			\int_0^\infty \left|\left(\partial_r+\tfrac{\lambda_k}{r}\right)u_k^-(r)\right|^2\,dr+
			\int_0^\infty \left|\left(\partial_r-\tfrac{\lambda_k}{r}\right)u_k^+(r)\right|^2\,dr \Bigg]
			\\
			&=
			\sum_{k\in \N} \Bigg[
			\int_0^\infty \left|\partial_r\left(r^{\lambda_k}u_k^-(r)\right)\right|^2 r^{-2\lambda_k}\,dr+
			\int_0^\infty\left|\partial_r\left(r^{-\lambda_k}u_k^+(r)\right)\right|^2 r^{2\lambda_k}\,dr \Bigg].
	\end{split}
	\end{equation}
Thanks to \cite[Proposition 2.4 $(i)$, $(ii)$]{CP2018} , and using that $u_k^\pm(0)=u_k^\pm(\infty)=0$, we have that for any $k\in \N$
	\begin{equation}\label{eq:apply-hardy-lambda}
	\begin{gathered}
		\int_0^\infty \left|\partial_r\left(r^{-\lambda_k}u_k^+(r)\right)\right|^2 r^{2\lambda_k}\,dr\geq 
			\left(\lambda_k-\tfrac{1}{2}\right)^2
			\int_0^\infty \frac{|u_k^+(r)|^2}{r^2}\,dr,\\
		\int_0^\infty\left|\partial_r\left(r^{\lambda_k}u_k^-(r)\right)\right|^2 r^{-2\lambda_k}\,dr\geq
			\left(\lambda_k+\tfrac{1}{2}\right)^2
			\int_0^\infty \frac{|u_k^-(r)|^2}{r^2}\,dr.
	\end{gathered}	
	\end{equation}
Since $\min_{k\in\N} \left(\lambda_k\pm\tfrac{1}{2}\right)^2 = \tfrac{(\pi-\omega)^2}{4\omega^2}$, combining  \eqref{eq:dec-sigma-grad-radial} with \eqref{eq:apply-hardy-lambda} we have that:
\[
	\begin{split}
	\int_{S_\omega} |\boldsymbol\sigma\cdot \nabla\psi|^2\,dx
		&\geq 
		\frac{(\pi-\omega)^2}{4\omega^2}
		\sum_{k\in \N} \Bigg[
		\int_0^\infty \frac{|u_k^+(r)|^2}{r^2}\,dr+
		\int_0^\infty \frac{|u_k^-(r)|^2}{r^2}\,dr\Bigg] \\
		& =
		\frac{(\pi-\omega)^2}{4\omega^2}
		\int_{S_\omega} \frac{|\psi|^2}{|x|^2}\,dx.
	\end{split}	
\]
	Finally, thanks to the fact that $\psi$ is supported far away from the origin, and that the curvature of a straight line is null, applying \cite[Equation (2.5)]{PVDB2021} we have that  
	\[
		\int_{S_\omega} |\boldsymbol\sigma\cdot \nabla\psi|^2\,dx=
		\int_{S_\omega} |\nabla\psi|^2\,dx.
	\]	
	This implies that the $H_\omega$-norm is equivalent to the $H^1$-norm on $S_\omega$ and so, by a density argument, we conclude the proof.
\end{proof}

From now on we denote by $H_\omega^{(D)}$ the self-adjoint operator $H_\omega$ defined in \eqref{eq:def-free-inf-mass} when $0<\omega\leq \pi$, and its distinguished self-adjoint extension (the unique one whose domain is included in the Sobolev space $H^{1/2}$) when $\pi<\omega\leq 2\pi$.
We recall a known result, namely \cite[ Proposition 1.12]{LTOB2018} that states that
\begin{equation}\label{eq:ess-spectrum-inf-mass}
		\sigma_{\mathrm{ess}}(H_\omega^{(D)})=(-\infty,-m]\cup[m,+\infty).
\end{equation}

We are ready now to prove \Cref{thm:ess-spect}.

\begin{proof}[Proof of \Cref{thm:ess-spect}]
We can always assume that $T$ verifying the property $\D(T)\subset H^s(S_\omega;\C^2)$ for $s<1/2$, see \Cref{rem:sobolev-regularity}. Indeed, if $\tilde{T}$ is an extension that does not verify this property, then $\tilde{T}$ is a finite rank perturbation of $T$  in the sense of resolvent differences and so they have the same essential spectrum.
We use the strategy of \cite[Section 4.3.4]{thaller},
exploiting the Weyl theorem \cite[Theorem XIII.14]{reedsimon4}: we prove that for any $z\in \C\setminus \R$ the operator $(T-z)^{-1}-(H_\omega^{(D)}-z)^{-1}$ is compact. Consequently, from \eqref{eq:ess-spectrum-inf-mass} we have that
	\[
		\sigma_{\mathrm{ess}}(T)=\sigma_{\mathrm{ess}}(H_\omega^{(D)})=(-\infty,-m]\cup[m,+\infty).
	\] 	

	For this purpose, let $\chi$ be defined as in \eqref{eq:def-chi} and, for any $n\in\N$, set $\chi_n(x):=\chi(|x|/n)$ and $\zeta_n(x):=1-\chi_n(x)$.
	Then
	\[
	\begin{split}
		(T-z)^{-1}-(H_\omega^{(D)}-z)^{-1}
			=&\,(T-z)^{-1}\chi_n-(H_\omega^{(D)}-z)^{-1}\chi_n\\
			 &+(T-z)^{-1}\zeta_n - \zeta_n(H_\omega^{(D)}-z)^{-1}\\
			 &+\zeta_n(H_\omega^{(D)}-z)^{-1}-(H_\omega^{(D)}-z)^{-1}\zeta_n\\
			=&: S^{(1)}_n+S^{(2)}_n+S^{(3)}_n.
	\end{split}
	\]
	
	We claim that $S^{(1)}_n=\left(\chi_n(T-\overline{z})^{-1}-\chi_n(H_\omega^{(D)}-\overline{z})^{-1}\right)^*$ is compact. Indeed, both operators $\chi_n(T-\overline{z})^{-1}$ and $\chi_n(H_\omega^{(D)}-\overline{z})^{-1}$ are compact since they are both bounded from $L^2(S_\omega;\C^2)$ to $H^{s}(S_\omega\cap B_{2n};\C^2)$ (being $B_{2n}$ the ball of radius $2n$) and $H^{s}(S_\omega\cap B_{2n};\C^2)$ is compactly embedded in $L^2(S_\omega;\C^2)$. 
	
	Let us analyse $S^{(2)}_n$. We observe that $V_n:=\frac{\nu}{|x|}\zeta_n\in L^\infty$, and so
		\[
	\begin{split}
		S^{(2)}_n=&
					-(T-z)^{-1}\left(\zeta_n(H_\omega-z)-(T-z)\zeta_n\right)(H_\omega-z)^{-1}\\
			=&-(T-z)^{-1}(V_n)(H_\omega-z)^{-1}\\
			&-(T-z)^{-1}(-i\boldsymbol\sigma\cdot\nabla\zeta_n)(H_\omega^{(D)}-z)^{-1}.
	\end{split}
	\]
	Using the fact that $||V_n||_{\infty}\leq \frac{C}{n}$ and $||-i\sigma\cdot\nabla\zeta_n||_{\infty}\leq \frac{C}{n}$, we can conclude that $S^{(2)}_n\to 0$ in the operator norm for $n\to+\infty$.
	
	Let us finally estimate $S^{(3)}_n$. Reasoning as above we have:
	\[
	 		S^{(3)}_n=
			-(H_\omega^{(D)}-z)^{-1}(-i\boldsymbol\sigma\cdot\nabla\zeta_n)(H_\omega^{(D)}-z)^{-1}\to 0\quad\text{in the operator norm for}\ n\to+\infty.	
	\]
	This concludes the first part of the proof.
Having proved that $\sigma_{\mathrm{ess}}(T) = (-\infty,-m] \cup [m,+\infty)$ one has immediately that $\sigma_{\text{d}}(T) \subset (-m,m)$. Recalling that $\sigma_{\text{d}}(T)$, the discrete spectrum of $T$, is the set of isolated eingevalues with finite multiplicity we note the following. Pick up $\lambda \in (-m,m)$, then there are two possibilities: either $\lambda$ is an eigenvalue of $T$ (which means that the eigenvalue  or there exists a neighbourhood of $\lambda$ that is contained in the resolvent set $\rho(T)$, that is $\lambda \in \rho(T)$. If this is the case and the operator $H_{\text{min}}$ is not essentially self-adjoint, since the dimension of deficiency subspaces is constant in the resolvent set, $\dim \ker(H^*_{\mathrm{min}}-\lambda) \geq 1$. Using Kre\u{\i}n-Vi\v{s}ik-Birman-Grubb extension scheme (see \cite[Theorem 2.13]{GM-book}) one can consider the self-adjoint extension $H_{\mathbbm{0}}-\lambda$ of $H_{\mathrm{min}}-\lambda$ corresponding to the Birman parameter $\mathbb{O}$ that has domain $\mathcal{D}(H_{\mathbb{O}}) \;=\; \mathcal{D}(H_{\text{min}})\dot{+} \ker ((H_{\text{min}})^*-\lambda)$. Shifting the operator by $\lambda$ doesn't change its domain and thus $\mathcal{D}(H_{\mathbb{O}})$ is a domain of a self-adjoint extension of $H_{\text{min}}$ with eigenvalue $\lambda$.
\end{proof}

\appendix
\section{Properties of the angular operator}
\label{sec:appendix-angular}
In this appendix we 
decompose the Hilbert space $L^2(S_\omega ; \C^2)$ in the direct sum of \emph{partial wave subspaces},
namely invariant subspaces for the action of the Dirac operator with a potential having
spherical symmetry. The topic is very well known so we give here a light presentation: for complete details the reader can see e.g.~\cite{LTOB2018,PVDB2021,CL20,FL21} or to \cite[Section 4.6]{thaller} for the analogous three dimensional reduction.

We use the standard notation for polar coordinates: for $x=(x_1,x_2)\in\R^2\setminus\{0\}$, 
\[
	\begin{array}{l}
		x_1=r\cos\theta,\\
		x_2= r\sin\theta,
	\end{array}
	\quad\text{being}\quad
	\begin{array}{l}
		r:=\sqrt{x_1^2+x_2^2}\in (0,+\infty),\\
		\theta:=\operatorname{sign}(x_2)\arccos\left(x_1/r \right)\in[0,2\pi).
	\end{array}
\]
For all $\psi \in L^2(S_\omega; \C)$, let $\varphi  = \varphi(r,\theta) : (0,+\infty) \times (0,\omega) \to \C$ be defined as follows
\[
	\varphi(r,\theta):= \sqrt{(x_1(r,\theta))^2 + (x_2(r,\theta))^2} \, \psi \left(x_1(r,\theta),x_2(r,\theta)\right)
	\quad \text{ for all } r \in (0,+\infty), \theta \in (0,\omega).
\]
The map $\psi \mapsto \varphi$ is a unitary map $L^2(S_\omega; \C) \to L^2((0,+\infty);\C)\otimes L^2((0,\omega);\C)$, since
\[
	\int_{S_\omega} |\psi(x)|^2\,\ud x\,=\,
		\int_0^\omega \int_0^{+\infty} |\varphi(r,\theta)|^2 \,\ud r\,\ud\theta.
\]
Repeating this reasoning for every component of the wave-function, we obtain the decomposition:
\begin{equation}\label{eq:equivL2}
	L^2(S_\omega;\C^2)\,\cong\,
	L^2((0,+\infty),\ud r)\otimes L^2((0,\omega);\C^2),
\end{equation}
where `` $\cong$ ''means unitarly equivalent. 

It is useful to express the Dirac operators in polar coordinates: setting
\[
	\boldsymbol{e_r}:=(\cos\theta,\sin\theta)=\frac{x}{r},\quad
	\boldsymbol{e_\theta}:=(-\sin\theta,\cos\theta)=\frac{\partial \boldsymbol{e_r}}{\partial\theta},
\]
we abbreviate  
\[
	\partial_r=\boldsymbol{e_r}\cdot{\nabla}
	\quad\text{and}\quad
	\partial_\theta=\boldsymbol{e_\theta}\cdot{\nabla}.
\]
By means of elementary computations, it is easy to see that
\[
	\boldsymbol{\sigma} \cdot \boldsymbol{e_r} \;=\; \begin{pmatrix}
		0 & e^{-\ii \theta} \\
		e^{\ii \theta} & 0
	\end{pmatrix}.
\]
and that the identity $\boldsymbol{\sigma}\cdot\boldsymbol{e_\theta} = i \boldsymbol{\sigma}\cdot\boldsymbol{e_r} \sigma_3$ holds. We obtain
\begin{equation}\label{eq:polar_dirac}
	-\ii\boldsymbol{\sigma}\cdot\nabla= 
		-\ii \boldsymbol{\sigma}\cdot \left(\boldsymbol{e_r} \partial_r + \frac{1}{r}\boldsymbol{e_\theta}\partial_\theta \right)
	 = -\ii \boldsymbol{\sigma}\cdot\boldsymbol{e_r} \left(\partial_r  +\frac{1}{2r}-\frac{K_\omega}{r}  \right)
\end{equation}
where $K_\omega$ is the \emph{spin-orbit operator}
\[
K_\omega := \frac{1}{2}\mathbbm{1} - i \sigma_3 \partial_\theta.
\]

In order to decompose appropriately $L^2((0,\omega);\C^2)$, we recall \cite[Lemma 2.4]{PVDB2021} about the properties of $K_\omega$.
\begin{proposition}[Properties of the spin-orbit operator]\label{prop:properties.spin-orbit}
Let $\omega\in(0,2\pi]$, $S_\omega$ as in \eqref{eq:defn.sector} and $\{\lambda_k\}_{k \in \N}$ as in \eqref{eq:def-lambda_k}.
Set 
\begin{equation}\label{eq:def-f_k}
	f_k^+ (\theta) \; := \; \frac{1}{\sqrt{2\omega}} \begin{pmatrix}
				e^{\ii(\lambda_k-\frac{1}{2}) \theta} \\ e^{-\ii(\lambda_k-\frac{1}{2}) \theta}
						\end{pmatrix}, \quad 
	f_k^- (\theta) \; := \; \frac{-\ii}{\sqrt{2\omega}} \begin{pmatrix}
				e^{-\ii(\lambda_k+\frac{1}{2}) \theta} \\ e^{+\ii(\lambda_k+\frac{1}{2}) \theta}
						\end{pmatrix},
		\quad \text{ for }\theta \in (0,\omega).
		\end{equation}
The spin-orbit operator with infinite mass boundary conditions
\[
	\begin{split}
		& K_\omega \;:=\; \frac{1}{2} \mathbbm{1}- \ii \sigma_3 \partial_\theta \, , \\
		& \mathcal{D}(K_\omega) \;:=\; 
			\big\{ 
				\phi=(\phi_1,\phi_2) \in H^1((0,\omega), \mathbb{C}^2) \, : \, 
				\phi_2(\omega)=-e^{\ii \omega} \phi_1(\omega) \, , \,
				\phi_1(0)=\phi_2(0) 
			\big\}\,
			\end{split}
\]
has the following properties:
	\begin{enumerate}[label=$({\roman*})$]
		\item $K_\omega$ is self-adjoint and has a compact resolvent;
		\item $\sigma(K_\omega)=\{ \lambda_k^{\pm}\}_{k \in \mathbb{N}}$;
		\item\label{item:ONB}
		$\{f^+_k, f^-_k\}_{k \in \mathbb{N}}$  is an orthonormal basis of eigenfunctions of $L^2((0,\omega);\mathbb{C}^2)$ with eigenvalues 
	$\{ \lambda_k^+,\lambda_k^-\}_{k \in \mathbb{N}}$; 
		\item $-\ii(\boldsymbol{\sigma} \cdot \boldsymbol{e_r} ) f^\pm_k = \pm  f^\mp_k $ for all $k \in \N$.
	\end{enumerate}
\end{proposition}

In the following proposition we finally decompose the space $L^2(S_\omega;\C^2)$ in partial wave subspaces.
\begin{proposition}[Decomposition in partial wave subspaces]\label{prop:decomposition.partialwave}
	Let $\omega\in (0,2\pi]$, $S_\omega$ be defined as in \eqref{eq:defn.sector} and for all $k\in \mathbb{N}$ let $\lambda_k^{\pm}$ as in \eqref{eq:def-lambda_k} and $f_k^{\pm}$ as in \eqref{eq:def-f_k}.
	Then, 
\[
	L^2(S_\omega;\C^2) 
	\,\cong\,
	\bigoplus_{k \in \N}
	\left[
	L^2((0,+\infty), \ud r)\otimes  \operatorname{span}\{f_k^+, f_k^-\}\right],
\]
i.e.~for any $\psi\in L^2(S_\omega;\C^2)$ there exists $\{(u^+_k,u^-_k)\}_{k\in\N}\in L^2((0,\infty),\ud r)\oplus L^2((0,\infty),\ud r)$ such that
	\begin{equation}\label{eq:rappr.psi.radial}
	\psi(r,\theta)=\frac{1}{\sqrt{r}}\sum_{k\in \N}
	\left[u_k^+(r) f_k^+(\theta)+ u_k^-(r) f_k^-(\theta)\right]
	\quad \text{ for a.a. } r\in (0,+\infty), \theta \in (0,\omega)
	\end{equation}
	and
\[
\Vert \psi\Vert_{L^2(S_\omega;\C^2)}^2 = 
\sum_{k \in \N} \left[ \Vert u_k^+\Vert_{L^2((0,+\infty),\ud r)}^2 + \Vert u_k^-\Vert_{L^2((0,+\infty),\ud r)}^2 \right].
\]
	\end{proposition}
	\begin{proof}
	The proof is immediate from \eqref{eq:equivL2} and \ref{item:ONB} of \Cref{prop:properties.spin-orbit}, .
	\end{proof}
Thanks to \Cref{prop:decomposition.partialwave}, it is possible to decompose the Dirac operator $H_\text{min}$ 
defined in \eqref{eq:def-H_min} 
as the direct sum of the one dimensional Dirac operators on the half-line with Coulomb potentials $h_{\nu,k}$ 
defined in \eqref{eq:def-h_k}. 
The following proposition implies \Cref{prop:radial.decomposition}.
\begin{proposition}\label{prop:UnitaryTransformation}
Let $\omega \in (0,2\pi]$, $S_\omega$ defined as in \eqref{eq:defn.sector}, $H_\text{min}$ defined as in \eqref{eq:def-H_min} and for all $k \in \N$, let $h_{\nu,k}$  be defined as in \eqref{eq:def-h_k}. Then, if $m=0$,
\[
		 H_{\operatorname{min}} \;\cong\; \bigoplus_{k \in \N} h_{\nu,k}\
\]
	where `` $\cong$ '' means that the operators are unitarily equivalent. 

In detail, for $\psi \in L^2(S_\omega;\C^2)$ there exists $\{(u^+_k,u^-_k)\}_{k\in\N}\subset L^2((0,\infty),\ud r)\oplus L^2((0,\infty),\ud r)$ such that  \eqref{eq:rappr.psi.radial} holds true and
	\begin{equation}
	\label{eq:char.dom.min}
	\psi \in \operatorname{dom} H_\text{min} \iff (u_k^+, u_k^-) \in \operatorname{dom} h_{\nu,k} = C^\infty_c\left ((0,+\infty);\C^2\right) \text{ for all }k \in \N.
	\end{equation}
Moreover
\begin{equation}\label{eq:compute.HV}
\left(D_0+\frac{\nu}{|x|}\mathbbm{1}_2\right) \psi (r,\theta) = \frac{1}{\sqrt{r}}\sum_{k\in \N}
	\left[\widetilde{u}_k^+(r) f_k^+(\theta)+ \widetilde{u}_k^-(r) f_k^-(\theta)\right], \quad \text{ for a.a. }r\in(0,+\infty), \theta \in(0,\omega),
\end{equation}
with
\begin{equation}\label{eq:compute.HV.2}
\begin{pmatrix} \widetilde{u}_k^+ \\ \widetilde{u}_k^- \end{pmatrix} = d_{\nu,k} \begin{pmatrix} u_k^+ \\ u_k^-\end{pmatrix}.
\end{equation}
\end{proposition}

%
\begin{proof}
The equivalence in \eqref{eq:char.dom.min} is immediate, since $f_k^\pm \in C_c^{\infty}([0,\omega];\C^2)$. 
Using \eqref{eq:polar_dirac} and the fact that
\[
\left(\partial_r + \frac{1}{2r}\right)\frac{1}{\sqrt{r}}  = 
\frac{1}{\sqrt{r}} \partial_r,
\]
we compute
\[
\begin{split}
\bigg(D_0+\frac{\nu}{|x|}&\mathbbm{1}_2\bigg) \psi(r,\theta)  = 
-\ii (\boldsymbol{\sigma} \cdot \boldsymbol{e_r}  )\left(\partial_r +\frac{1}{2r} -\frac{K_\omega}{r}\right) \frac{1}{\sqrt{r}}
\sum_{k \in \N} 
\left[u_k^+(r) f_k^+(\theta)+ u_k^-(r) f_k^-(\theta)\right]
\\ & = \frac{-\ii (\boldsymbol{\sigma} \cdot \boldsymbol{e_r})}{\sqrt{r}} 
\sum_{k \in \N} 
\left[\partial_r u_k^+(r) f_k^+(\theta)+ \partial_r u_k^-(r) f_k^-(\theta)
-\frac{\lambda_k}{r} u_k^+(r) f_k^+(\theta)+ \frac{\lambda_k}{r} u_k^-(r) f_k^-(\theta)\right]
\\ & = 
\frac{1}{\sqrt{r}} 
\sum_{k \in \N} 
\left[\partial_r u_k^+(r) f_k^-(\theta)- \partial_r u_k^-(r) f_k^+(\theta)
-\frac{\lambda_k}{r} u_k^+(r) f_k^-(\theta) -\frac{\lambda_k}{r} u_k^-(r) f_k^+(\theta)
\right],
\end{split}
\]
From this \eqref{eq:compute.HV} and \eqref{eq:compute.HV.2} follow.
\end{proof}

\section*{Acknowledgements}

The idea of this project came out during the ``International Congress in Mathematical Physics'' held in Geneva in August 2021. We would like to express our gratitude to the organizers of the conference. This work was partially supported by the Istituto Nazionale di Alta Matematica (INdAM), the MIUR-PRIN 2017 project MaQuMA cod.~2017ASFLJR. This work was partially developed when F.~P.~was  employed at CNRS \& CEREMADE - Université Paris Dauphine and he was partially supported by the project ANR-17-CE29-0004 molQED of the Agence Nationale de la Recherche. He has also received funding from the and by European Research Council (ERC) under the European Union’s Horizon 2020 research and innovation programme (grant agreement MDFT No.~725528 of Mathieu Lewin).

\end{document}